\documentclass[11 pt]{amsart}

\usepackage{amsmath,amsthm,amssymb,amscd}

\newtheorem{thm}{Theorem}[section]
\newtheorem{cor}[thm]{Corollary}
\newtheorem{lem}[thm]{Lemma}
\newtheorem{ass}[thm]{Assumption}

\newtheorem{prop}[thm]{Proposition}

\theoremstyle{definition}
\newtheorem{df}[thm]{Definition}
\theoremstyle{remark}
\newtheorem{rem}[thm]{Remark}
\newtheorem{ex}{Example}
\numberwithin{equation}{section}

\def\be#1 {\begin{equation} \label{#1}}
\newcommand{\ee}{\end{equation}}

\newcommand{\mb}{\medskip\noindent}

\newcommand{\R}{\mathbb R}

\newcommand{\N}{\mathbb N}
\newcommand{\Z}{\mathbb Z}

\newcommand{\BMO}{\textrm{BMO}}

\def \vsp {\vspace{6pt}}

\DeclareMathOperator{\supp}{supp}
\DeclareMathOperator*{\esssup}{ess\ sup}
\newcommand{\eps}{\ensuremath{\varepsilon}}
\newcommand{\abs}[1]{\left|#1\right|}
\newcommand{\norm}[1]{\left\|#1\right\|}

\def\Xint#1{\mathchoice
   {\XXint\displaystyle\textstyle{#1}}%
   {\XXint\textstyle\scriptstyle{#1}}%
   {\XXint\scriptstyle\scriptscriptstyle{#1}}%
   {\XXint\scriptscriptstyle\scriptscriptstyle{#1}}%
   \!\int}
\def\XXint#1#2#3{{\setbox0=\hbox{$#1{#2#3}{\int}$}
     \vcenter{\hbox{$#2#3$}}\kern-.5\wd0}}

\def\aver#1{\Xint-_{#1}}

\begin{document}

\title{Pseudodifferential operators associated with a semigroup of operators.}

\day=10 \month=12 \year=2012

\date{\today}

\author{Fr\'ed\'eric Bernicot}
\address{Fr\'ed\'eric Bernicot - CNRS - Universit\'e de Nantes \\ Laboratoire Jean Leray \\ 2, rue de la Houssini\`ere \\
44322 Nantes cedex 3, France }
\email{frederic.bernicot@univ-nantes.fr}

\author{Dorothee Frey}
\address{Dorothee Frey - Mathematical Sciences Institute, John Dedman Building, Australian National University, Canberra ACT 0200, Australia}
\email{dorothee.frey@anu.edu.au}



\thanks{The first author is supported by the ANR under the project AFoMEN no. 2011-JS01-001-01. 
The second author is supported by the Australian Research Council Discovery grants DP110102488 and DP120103692.}

\date{\today}

\begin{abstract} Related to a semigroup of operators on a metric measure space, we define and study pseudodifferential operators (including the setting of Riemannian manifold, fractals, graphs ...). Boundedness on $L^p$ for pseudodifferential operators of order $0$ are proved. Mainly, we focus on symbols belonging to the class $S^0_{1,\delta}$ for $\delta\in[0,1)$. For the limit class $S^0_{1,1}$, we describe some results by restricting our attention to the case of a sub-Laplacian operator on a Riemannian manifold.
\end{abstract}

\maketitle

\begin{quote}
\footnotesize\tableofcontents
\end{quote}

In this paper, we define and study pseudodifferential operators on metric measure spaces endowed with a non-negative, self-adjoint operator, which generates a semigroup satisfying certain off-diagonal estimates. 

Pseudodifferential operators are now well-known in the Euclidean setting and have been powerful tools in several situations: propagation of regularity for nonlinear PDEs with the paralinearization, microlocal analysis, ... \\
For operators related to the H\"ormander symbolic classes $S^m_{\rho,\delta}$, the theory is based on 
\begin{itemize}
 \item Boundedness of the corresponding pseudodifferential operators on Lebesgue and Sobolev spaces;
 \item A functional calculus for associated operators with symbolic calculus, which can be described by very convenient ``rules'' on the principal part of the symbols.
\end{itemize}
Throughout this theory, very well developed in the Euclidean setting, the Fourier transform is a crucial tool to study and define the symbolic classes. The theory can be extended to some geometric groups, where a kind of Fourier transform exists (see \cite{BFG} for Heisenberg groups, \cite{GS} for H-type Carnot groups, ...).

Since this theory is now well-known and many works have contributed to the development, we do not detail this situation. We just emphasize that pseudodifferential operators of order $0$ associated with a symbol $\sigma \in S^0_{1,\delta}$ always have a kernel representation with a kernel satisfying the standard properties. Thus, 
\begin{itemize}
 \item If $\delta<1$, then $\sigma(x,D)$ is a Calder\'on-Zygmund operator, and is $L^p$-bounded for every $p\in (1,\infty)$;
 \item If $\delta=1$, then $\sigma(x,D)$ is a Calder\'on-Zygmund operator, and, according to the $T(1)$-theorem, $L^p$-bounded for every $p\in (1,\infty)$ if and only if $[\sigma(x,D)]^*(1) \in \BMO$.
\end{itemize}
The classical class $S^m_{1,0}$ (and more generally $S^m_{1,\delta}$ for $\delta\in(0,1)$) satisfies very nice symbolic rules for some pseudodifferential calculus. However, the limit class $S^m_{1,1}$(which can be seen more exotic) is very important, since it naturally appears in the context of paraproducts (and therefore in the paralinearization argument developed by Bony in \cite{Bony}).

After that, people were naturally interested in extending this theory for manifolds. This was done, on smooth manifolds since there, where one can locally use charts to reduce the problem to the one in the Euclidean space. \\
More recently, several works are concerned with defining and studying pseudodifferential operators in a context where no such reduction hold: Riemannian manifolds, graphs, fractal sets etc. See e.g. \cite{IRS}.

Our aim in this paper is to describe a very weak structure, which allows us to define a suitable pseudodifferential calculus. We consider a space of homogeneous type $X$, and assume that this space is equipped with a Sobolev embedding and a Poincar\'e inequality. That is, we assume that there is a non-negative, self-adjoint operator $\Delta$, densely defined on $L^2(X)$, that satisfies a Sobolev embedding of the form \eqref{eq:sobolev}. 
We moreover assume that there is another operator $L$, that is non-negative, self-adjoint on $L^2(X)$, and satisfies $L^p$ off-diagonal estimates of the form (\ref{ass-Lq-offidag}) for $p$ in some interval $(p_0,p_0')$.
Let us emphasize, that we neither assume kernel estimates on the semigroup $(e^{-t\Delta})_{t>0}$ nor on $(e^{-tL})_{t>0}$, but instead work with the more general concept of Davies-Gaffney estimates and $L^p$ off-diagonal estimates, also called generalized Gaussian estimates. We discuss possible examples of operators $\Delta$ and $L$ in Section 2. \\
Then, we define a version of H\"ormander class $S^0_{1,\delta}$ for $\delta\in[0,1]$ associated with $L$. We investigate symbols, that are functions $\sigma :X \times \R \to \R$, sufficiently smooth (where the smoothness in ``$x$'' is measured in term of $\Delta$), and define via functional calculus (see Subsubsection \ref{subsub}) a corresponding pseudodifferential operator $T_\sigma:=\sigma(\,.\,,L)$ as 
\begin{equation} \label{def-psdo}
 T_\sigma(f) := x \mapsto \sigma(x,L)[f](x).
\end{equation}
We then aim to study $L^p$-boundedness of these operators. Since in our setting, the operator $\sigma(x,L)$ is in general not a Calder\'on-Zygmund operator, we have to adapt the methods used in the study of classical pseudodifferential operators. 
We first study $L^p$-boundedness for the classes $S^0_{1,\delta}$ with $\delta \in [0,1)$. Here, we combine the classical method of freezing coefficients and decomposition of symbols into elementary symbols \cite{CM}, and combine these with recent results on spectral multipliers for the operator $L$. We then obtain

\begin{thm} Assume that the ambiant space satisfies a Sobolev inequality (Assumption \ref{ass:sobolev}) and a symbol $\sigma\in S^0_{1,\delta}$ for $\delta\in[0,1)$. The pseudodifferential operator $\sigma(x,L)$ is bounded on $L^p$ for every $p\in(p_0,p_0')$ in the two following situations: if $\delta=0$, or if $\delta \in(0,1)$ with an extra weak Poincar\'e inequality (Assummption \ref{ass:poincare}).
\end{thm}

In generalization of Calder\'on-Zygmund kernel estimates, we can show that we have off-diagonal estimates for $T_\sigma\psi_t(L)$, where $\psi_t$ is a smooth function adapted to the scale $t^{-1}$. See Proposition \ref{prop:can} below. 
Such operators also satisfy weighted $L^p$-boundedness and act continuously on Sobolev spaces (defined relatively to $\Delta$ and $L$ by Bessel potentials).

We emphasize that this result holds in a very general context on a metric space. In particular, it allows us to regain and improve some existing recent results about pseudodifferential operators on fractal sets, cf. \cite{IRS}.\\
At the same time, we obtain a result on symbols, that are defined by a different operator than the one related to the structure of the underlying metric space. 
Similar generalizations in this direction were  e.g. already obtained in \cite{PortalStrkalj}, where pseudodifferential operators acting on UMD Banach spaces are considered. There, the symbol classes are defined via R-boundedness conditions, which are closely related to the kind of off-diagonal estimates we assume, cf. \cite{Kunstmann}, Theorem 2.2.\\

For the symbol class $S_{1,1}$, the situation is more difficult. We restrict ourselves to the case of a Riemannian manifold with a sub-Laplacian operator $\Delta=L$, where a Leibniz rule is available. In the classical theory of pseudodifferential operators, a main tool in the consideration of the symbol class $S^0_{1,1}$ is the $T(1)$-Theorem by David and Journ\'e \cite{DavidJourne}. We substitute this theorem by a $T(1)$-Theorem, that is associated with the operator $L$. For an exact statement of this result, we refer to \cite{B} in the setting of Poisson kernel bounds for $(e^{-tL})_{t>0}$, and to  \cite{FreyKunstmann} in the more general setting of $L^p$ off-diagonal bounds as described above.
We then prove

\begin{thm} Let $(X,d,\mu)$ be a Riemannian manifold with a sub-Laplacian structure with $L^{p_0}-L^{p_0'}$ off-diagonal decay of the semigroup (Assumption \ref{ass1}), $L^2-L^2$ estimates on the gradient of the semigroup (Assumption \ref{ass2}) and a Poincar\'e inequality (Assumption \ref{ass3}). Let a symbol $\sigma \in S^0_{1,1}$. Then, the pseudodifferential operator $\sigma(x,L)$ is $L^2$-bounded if and only if $[\sigma(x,L)]^*(1) \in \BMO_{L}$. 
\end{thm}

Moreover, in \cite{BBR}, algebra properties have been studied for Sobolev spaces (defined by Bessel potential) associated with a semigroup of operators. That corresponds to boundedness of pointwise product in such Sovolev spaces. It is then natural to extend the study for other operations. Here, pseudodifferential operators are shown to be bounded in these scales of Sobolev spaces.

\section{The setting}

\subsection{Space of homogeneous type with Sobolev embedding}

 Let $(X,d)$ be a metric space and let $\mu$ be a non-negative Borel measure on $X$ with  the doubling property: there exists a constant $A_1 \geq 1$ such that for all $x \in X$ and all $r>0$
 \begin{align*}  \mu(B(x,2r)) \leq A_1 \mu(B(x,r)) < \infty, \end{align*}
 where we set $B(x,r):=\{y \in X\,:\, d(x,y)<r\}$.

In particular, $(X,d,\mu)$ is a space of homogeneous type.

It is then well-known that there exists a constant $A_2>0$ and a dimension $n>0$ such that for all $\lambda \geq 1$, for all $x \in X$ and all $r>0$
 \begin{align} \label{doublingProperty2}
 		\mu(B(x,\lambda r)) \leq A_2 \lambda^n \mu(B(x,r)).
 \end{align}
As a consequence, a ball $B(x,tr)$ can be covered by $c t^n$ balls of radius $r$, uniformly in $x\in X$, $r>0$ and $t>1$. Moreover, there also exist constants $C$ and $D$, with $0 \leq D\leq n$, such that 
\begin{align*}
 	\mu(B(y,r)) \leq C \left(1+\frac{d(x,y)}{r}\right)^D \mu(B(x,r))
\end{align*}
uniformly for all $x,y \in X$ and $r>0$.

For a ball $B \subseteq X$ we denote by $r_B$ the radius of $B$ and set
\begin{equation} \label{annuli}
 	S_0(B):=B \qquad \text{and} \qquad S_j(B):= 2^jB \setminus 2^{j-1} B \quad \text{for} \; j=1,2,\ldots,
\end{equation}
where $2^jB$ is the ball with the same center as $B$ and radius $2^jr_B$.\\
For a measurable function $f$ and a ball $B \subseteq X$, we abbreviate $\aver{B} f\,d\mu = \mu(B)^{-1}\int_B f\,d\mu$.\\

We will assume that $X$ is equipped with a Sobolev embedding.

\begin{ass}[Sobolev embedding] \label{ass:sobolev} There exists a non-negative, self-adjoint operator $\Delta$ on $X$, densely defined on $L^2(X)$ (i.e. its domain ${\mathcal D}(\Delta):= \{f\in L^2(X),\ \Delta(f)\in L^2(X)\}$
is dense in $L^2(X)$), an integer $M_0$ and $\kappa>0$ such that for all $t>0$ and all balls $B$ of radius $r=t^{1/2}$, we have
\begin{equation} \label{eq:sobolev}
 \|f\|_{L^\infty(B)} \lesssim \sum_{i\geq 0} 2^{-i\kappa} \ \left(\aver{2^i B}\left|\left(I+t\Delta\right)^{M_0} f \right| d\mu \right),
 \end{equation}
uniformly in $f\in \cap_{N=0}^{M_0} {\mathcal D}(\Delta^N)$.
\end{ass}

\begin{rem} \emph{(i)} Here, we implicitly assume that $\Delta$ is an operator of order $2$ (since we consider $t=r^2$). All our current results can easily be extended if $\Delta$ is of another order. For better readability, we choose this order $2$, but the main idea is to show that this operator may be chosen independently of the other operator $L$ (appearing in the next subsection).\\
\emph{(ii)} The $L^1$ norm appearing in \eqref{eq:sobolev} may  be relaxed to a $L^p$ norm for some $p>1$, depending on the assumptions on the operator $L$ (see the application of \eqref{eq:sobolev} in the proof of Lemma \ref{lemma2} below).
\end{rem}

\begin{ex} Let us point out that if $\Delta$ generates a semigroup $(e^{-t\Delta})_{t>0}$, whose kernel $K_t$ satisfies the pointwise estimate
$$ \left|K_{r^2}(x,y) \right| \lesssim \mu(B(x,r)) \left(1+ \frac{d(x,y)}{r} \right)^{-M}$$
for a sufficiently large exponent $M$, then $\Delta$ satisfies a Sobolev embedding and Assumption \ref{ass:sobolev} is satisfied \cite[Proposition 3.10]{B}.
 \end{ex}

We will also use a weak version of Poincar\'e inequality for the symbols in $S_{1,\delta}$ with $\delta<1$. 

\begin{ass}[Generalized Poincar\'e inequality] \label{ass:poincare} The previous self-adjoint operator $\Delta$ satisfies: for $\kappa>0$ and a large enough integer $M>0$, for all $t>0$ and all balls $B$ of radius $r=t^{1/2}$, we have
$$ \|f - \aver{B} f \|_{L^\infty(2B)} \lesssim r_B \sum_{i\geq 0} 2^{-i\kappa} \ \left(\aver{2^i B}\left|\Delta \left(I+t\Delta\right)^{M} f \right|^2 d\mu \right)^{\frac{1}{4}} \left(\aver{2^i B}\left|\left(I+t\Delta\right)^{M}  f \right|^2 d\mu \right)^{\frac{1}{4}} ,$$
uniformly in $f\in \cap_{N=0}^{M} {\mathcal D}(\Delta^N)$.
\end{ass}
This assumption is weaker than a Poincar\'e type inequality, since it does not require a structure associated with a gradient operator.

\subsection{Symbols associated with self-adjoint operator}

\subsubsection{Self-adjoint operators}

We consider a non-negative, self-adjoint operator $L$ on $L^2(X)$. Due to the spectral theorem, $L$ has a bounded Borel functional calculus on $L^2(X)$.
Let $m>1$ be a fixed constant, representing the order of the operator $L$. 
We assume that the analytic semigroup $(e^{-tL})_{t>0}$, generated by $-L$, satisfies Davies-Gaffney estimates, and, moreover, $L^{p_0}-L^{p_0'}$ off-diagonal estimates for some $p_0 \in [1,2)$:

\begin{ass}[$L^{p_0}-L^{p_0'}$ off-diagonal estimates] \label{ass-Lq-offidag} 
There exist $p_0\in[1,2)$ and constants $C,c>0$ such that for arbitrary balls $B_1,B_2$ of radius $r=t^{1/m}>0$
\begin{equation} \label{eq:Lp0}
 	\|e^{-tL} \|_{L^{p_0}(B_1) \to L^{p_0'}(B_2)} \leq C \mu(B_1)^{\frac{1}{p_0'}-\frac{1}{p_0}} \ e^{-c\left(\frac{d(B_1,B_2)}{t^{1/m}}\right)^{\frac{m}{m-1}}},
\end{equation}
where $\frac{1}{p_0}+\frac{1}{p_0'}=1$.
\end{ass}

\begin{rem}
Assumption \ref{ass-Lq-offidag} implies that  $\{e^{-tL}\}_{t>0}$ satisfies $L^p$-$L^q$ off-diagonal estimates for all $p_0 \leq p \leq q \leq p_0'$, and therefore in particular Davies-Gaffney estimates (i.e. $L^2-L^2$ off-diagonal estimates).  Using the Phragm\'en-Lindel\"of theorem, one can also show extensions to complex times. See e.g. \cite{Blunck}, Theorem 2.1.
\end{rem}

\subsubsection{Class of symbols} \label{subsub}

Since $L$ is self-adjoint, we know that we have a bounded real functional calculus. 
We define $C^\infty(X \times (0,\infty))$ as the set of measurable functions $\sigma\in L^\infty(X \times (0,\infty))$ such that for every integer $j_1 \geq 0$ and $j_2 \in {\mathbb N}+\{0,\frac{1}{2}\}$, the map   
$$ (x,\xi) \mapsto \partial_\xi^{j_1} \Delta^{j_2}[ \sigma(\,.\,,\xi)](x)$$
is continuous.

So let us define our classes of symbols:

\begin{df} Let $\rho,\delta \in[0,1]$ and $s\geq 0$, we set $S^{s}_{\rho,\delta}:=S^{s}_{\rho,\delta}(L)$ the set of symbols $\sigma\in C^\infty(X \times (0,\infty))$ such that
 $$ \forall \alpha,\beta\geq 0, \qquad \left| \partial_\xi^\beta \Delta^\alpha \sigma(\,.\,,\xi) [x] \right| \lesssim (1+|\xi|)^{\frac{s}{m}-\rho \beta + \frac{2}{m}\delta \alpha}.$$
\end{df}

Let us first precise how can we rigorously define pseudodifferential operators. So let $\sigma\in S^s_{\rho,\delta}$ be a symbol. Then as explained in Lemma \ref{lem:decom}, this symbol may be split into a fast decaying sum of elementary symbols, each of them beeing tensorial product of the form 
$$ \tau(x,\xi):=\int_0^1 \gamma_t(x) \psi_t(\xi) \frac{dt}{t}.$$
To such a symbol, the functions $\psi_t$ beeing smooth and compactly supported, we know that $\psi_t(L)$ is well-defined and maps $L^2$ into $L^2$. Consequently, for every $\eps>0$ the operator defined by
$$ \tau(x,L)^\eps:= f \rightarrow \left( x \to \int_\eps^1 \gamma_t(x) \psi_t(L)[f](x) \frac{dt}{t}\right)$$
is well-defined for $f\in L^2$. Obtaining uniform bounds with respect to $\eps>0$ in $L^2$ (or in $L^p$, which will be the aim of our main results) and using that for every test-functions $f\in {\mathcal D}(L)$ we have $\|\psi_t f\|_{L^2} \lesssim t$, we also may define the pseudodifferential operator as follows: for every $f\in{\mathcal D}(L)$
$$ \tau(x,L)(f) = \lim_{\eps\to 0}  \tau(x,L)^\eps (f).$$

Let us also refer to \cite[Lemma A.1]{Nier} and \cite[Appendix]{IP} for a similar reasoning, by replacing the decomposition in elementary symbols with the Dynkin-Helffer-Sj\"ostrand formula. Indeed, if $\sigma(x,\xi)$ is compactly supported in $\xi\in[\eps, \eps^{-1}]$, then we may define 
$$ \sigma(x,L) := \frac{i}{2\pi} \int_{\mathbb C} \overline{\partial} \tilde{\sigma}(x,z) (z+L)^{-1} \, d\overline{z} \wedge dz,$$
where $\tilde{\sigma}$ is an-almost analytic extension of $\sigma$. Here, the idea is the same: we decompose the symbol as a tensorial product of $x$-variable function and a simple multiplier (here involving the resolvant).

\section{Examples}

\subsection*{Riemannian manifold}
In the case of a doubling Riemannian manifold $(X,d,\mu)$ satisfying a $(P_2)$ Poincar\'e inequality, the non-negative Laplacian $\Delta$ satisfies the Sobolev inequality (Assumption \ref{ass:sobolev}), since its heat kernel has Gaussian estimates. Moreover, the Poincar\'e inequality $(P_2)$ implies Assumption \ref{ass:poincare}. Let us sketch the proof of this claim. Indeed we have $L^2-L^2$ off-diagonal estimates for $\nabla \Delta^{-1/2} (1+t\Delta)^{-M}$ with a large enough exponent $M$ (see \cite{ACDH}) which allows us with Poincar\'e inequality and Sobolev inequality, to bound quantity $\|f-\aver{B} f \|_{L^\infty(2Q)}$ by a fast decaying serie with local quantities of the form $\left\| \Delta^{\frac{1}{2}}(f) \right\|_{L^2(2^i B)}$. Each of them are bounded by $\left\| \Delta(f) \right\|_{L^2(2^i B)}^{\frac{1}{2}} \left\| \Delta(f) \right\|_{L^2(2^i B)}^{\frac{1}{2}}$ which proves Assumption \ref{ass:poincare}.

\subsection*{Euclidean case with divergence form operator}
In the Euclidean case: we may choose $X$ as $\R^n$ or a doubling open set of $\R^n$ and so Assumptions \ref{ass:sobolev} and \ref{ass:poincare} are satisfied with the Euclidean Laplacian $\Delta$. Consider a homogeneous elliptic operator $L$ of order $m = 2k$ in $\R^n$ defined by
$$ L(f) := (-1)^k \sum_{|\alpha|=|\beta|=k} \partial^\alpha (a_{\alpha,\beta} \partial ^\beta f),$$
with bounded complex coefficients $a_{\alpha,\beta}$ with $a_{\beta,\alpha}=\overline{a_{\alpha,\beta}}$. Then $L$ is a self-adjoint of order $m$ and its heat semigroup satisfies pointwise estimate ($p_0=1$) if we are in one of the following situation: 
\begin{itemize}
 \item The coefficients $a_{\alpha,\beta}$ are real-valued (see \cite[Theorem 4]{AT});
 \item The coefficients are complex and $n \leq 2k = m$ (see \cite[Section 7.2]{A});
 \item The coefficients are H\"older continuous \cite{AMT}.
\end{itemize}
Else, we only know that there exists $p_0\in[1,2)$ such that Assumption (\ref{ass-Lq-offidag}) holds.\\

\subsection*{Fractal sets and infinite trees}
Assume that $X$ is a p.c.f. fractal set or a highly symmetric Sierpinski carpet. Then using the self-similar structure, it is well-known that we may build a Dirichlet energy ${\mathcal E}$ and then obtain a non-negative Laplace operator $\Delta$ such that
$$ {\mathcal E}(f,g) = \int_X \Delta(f) g d\mu.$$
Defining the effective resistance metric $d(x,y)$ by
$$ d(x,y)^{-1} = \min\{ {\mathcal E}(u,u),\ u(x)=0,\ u(y)=1\},$$
the space $(X,d,\mu)$ is a space of homogeneous type.
Then it is well-known that the operator generates a heat semigroup satisfying sub-Gaussian estimates, so Assumption (\ref{ass:sobolev}) is satisfied.

Let us examine (\ref{ass:poincare}). By definition of the resistance metric, it is obvious that for a function $f\in{\mathcal D}(\Delta)$, we have
\begin{align}
 \left|f(x) -f(y) \right| & \lesssim d(x,y)^{\frac{1}{2}} {\mathcal E}(f,f)^{\frac{1}{2}} \nonumber \\
 & \lesssim d(x,y)^{\frac{1}{2}} \left(\int_X \Delta(f) f d\mu \right)^{\frac{1}{2}} \nonumber \\
 & \lesssim d(x,y)^{\frac{1}{2}} \|\Delta(f)\|_{2}^{\frac{1}{2}} \|f\|_{2}^{\frac{1}{2}}.\label{eq:p}
\end{align}
Then we may localize, so if $x,y$ are in  a ball $B$ of radius $r$, then we choose a bump function $\phi$ (which equals $1$ on $B$ and $0$ in $(2B)^c$). We refer the reader to \cite{RST} for  an abstract construction and existence of such functions. Then we apply (\ref{eq:p}) to $f=f\phi$ and we can prove local estimates with a an extra term in $d(x,y)^{\frac{1}{2}}$. This is not sufficient to get our Assumption (\ref{ass:sobolev}), where we need $d(x,y)$.
In \cite{BCK}, the authors proved that our desired Poincar\'e inequality is closely related to lower-estimates of the heat kernel or to a Harnack inequality. More precisely, the proof of \cite[Lemma 2.3 (1)]{BCK} shows how to improve the previous reasoning to get our weak Poincar\'e inequality. 
Indeed, it is well-known (\cite{BBK,BCK}) that the Poincar\'e inequality is equivalent to two-sided sub-Gaussian estimates for the heat kernel.
In \cite{BBK,BCK}, some examples of fractal graphs, trees, ... are studied in detail and our weak Poincar\'e estimates hold.

In such situations, our results allow us to obtain boundedness for pseudodifferential operators associated with symbols in the class $S^0_{1,\delta}$ for $\delta\in[0,1)$. We then regain some recent results in \cite{IRS} where the class $S^0_{1,0}$ was studied.

\section{Preliminary results}

We start with a result on pseudo-differential operators with constant coefficients. We show that a pseudo-differential operator with constant coefficients (i.e. a spectral multiplier) satisfies off-diagonal estimates at scale $r$, if the underlying symbol is adapted to the scale $r$, as described in \eqref{m-cond1} below. The off-diagonal estimates are of polynomial order, where the order depends on the behaviour of the symbol at $0$.\\ 
The result is already implicitely contained in e.g. \cite{KunstmannUhl} or \cite{Anh}. For convenience of the reader, we give the proof here. 

We use the following 
\emph{partition of unity}:
Let $\eta \in C^{\infty}(0,\infty)$ with $\eta(\xi)=1$ for $0<\xi \leq 1$, and $\eta(\xi)=0$ for $\xi \geq 2$. Define $\delta \in C^{\infty}(0,\infty)$ by $\delta(\xi)=\eta(\xi)-\eta(2\xi)$. Then $\supp \delta(2^{-j} \,\cdot\,) \subseteq  [2^{j-1},2^{j+1}]$ and
\[
		1 = \sum_{j=-\infty}^{\infty} \delta(2^{-j}\xi), \qquad \xi >0.
\]

\begin{lem} \label{lemma1}
Let $p \in (p_0,p_0')$. Let $\nu >n\abs{\frac{1}{p}-\frac{1}{2}}$ and $N>\frac{\nu}{m}$, let $r>0$. Let $F$ be a smooth function on $[0,\infty)$ with
\begin{equation} \label{m-cond1}
 \left|\partial^\beta_\xi F(\xi)  \right| \lesssim \min(1,(r^m|\xi|)^N) |\xi|^{-\beta}
\end{equation}
for all indices $0 \leq \beta \leq \lfloor \nu \rfloor +1$. 
Then there exists a constant $C>0$, independent of $r>0$, such that for all balls $B_1,B_2$ of radius $r$ and all $f \in L^p(X)$ with $\supp f \subseteq B_1$
\begin{align*}
 	\norm{F(L)f}_{L^p(B_2)} \leq C \left(1+\frac{d(B_1,B_2)}{r}\right)^{-\nu} \norm{f}_{L^p(B_1)}.
\end{align*}
\end{lem}

\vsp

\begin{proof}
Let $B_1,B_2$ be two balls of radius $r>0$ in $X$, let $f \in L^p(X)$ with $\supp f \subseteq B_1$.
The estimate 
\[
		\norm{F(L)f}_{L^p(B_2)} \leq C \norm{f}_{L^p(B_1)}.
\]
follows for $p=2$ directly from the fact that $L$ is self-adjoint, and thus has a bounded Borel functional calculus on $L^2(X)$.
For $p \neq 2$, the estimate follows from spectral multiplier results (see e.g. \cite{KunstmannUhl}, Theorem 5.4 b)), where we use that $\nu >n\abs{\frac{1}{p}-\frac{1}{2}}$.

This yields the desired estimate for $d(B_1,B_2)\leq r$. For $d(B_1,B_2)>r$, on the other hand, we denote $G(\xi):=F(\xi^m)$, 
and split $G$ into
\[
		G(\xi) = \sum_{j=-\infty}^{\infty} G(\xi) \delta(2^{-j}\xi) = \sum_{j=-\infty}^{\infty} G_j(\xi), \qquad \xi \neq 0.
\]
Then $\supp G_j \subseteq [2^{j-1}, 2^{j+1}]$ for all $j \in \Z$. Denote $\eps:=\lfloor \nu \rfloor +1 -\nu >0$.
The application of (a slight modification of) \cite{KunstmannUhl}, Lemma 4.10, to $G_j$ yields
\begin{equation} \label{m-est-eq1}
		\|G_j(\sqrt[m]{L})f\|_{L^p(B_2)} \leq C \frac{2^{-j\nu}}{d(B_1,B_2)^\nu} \|\delta_{2^{j+1}}G_j\|_{W^{\infty}_{\nu+\eps}}  \|f\|_{L^p(B_1)}.
\end{equation}
Observe now that for each $j \in \Z$, we have $\supp \delta_{2^{j+1}}G_j \subseteq [\frac{1}{4} ,1]$ and, due to assumption \eqref{m-cond1},
 $\|\delta_{2^{j+1}}G_j\|_{W^{\infty}_{\nu+\eps}} \lesssim \min(1,(2^jr)^{mN})$.
Plugging this into \eqref{m-est-eq1}, we obtain for each $0<N'\leq N$
\begin{equation} \label{m-est-eq2}
			\|G_j(\sqrt[m]{L})f\|_{L^p(B_2)} \lesssim  \frac{2^{-j\nu}}{d(B_1,B_2)^\nu} (2^jr)^{mN'} \|f\|_{L^p(B_1)}.
\end{equation}
If $2^{j}r<1$, we choose $N'=N$ and write
\[
 		\frac{2^{-j\nu}}{d(B_1,B_2)^\nu} (2^jr)^{mN} = \frac{r^{\nu}}{d(B_1,B_2)^\nu} (2^jr)^{mN-\nu},
\]
if, on the other hand, $2^{j}r \geq 1$, we choose $N'<\frac{\nu}{m}$ and write
\[
 		\frac{2^{-j\nu}}{d(B_1,B_2)^\nu} (2^jr)^{mN'} = \frac{r^\nu}{d(B_1,B_2)^\nu} (2^jr)^{-(\nu-mN')}.
\]
Finally, observe that there exists a constant $C>0$, independent of $r$, such that
\[
 		\sum_{j \in \Z: \, 2^jr<1} (2^jr)^{mN-\nu} + \sum_{j \in \Z: \, 2^jr \geq 1} (2^jr)^{-(\nu-mN')} <C,
\]
since $N>\frac{\nu}{m}$ and $N'<\frac{\nu}{m}$.
Summing over $j$ in \eqref{m-est-eq2} yields the assertion.
\end{proof}

The next lemma gives an almost orthogonality condition for multipliers of the above form, similar to e.g. \cite{AuscherMcIntoshRuss}, Lemma 3.7.

\begin{lem} \label{schur-lemma}
Let $p \in (p_0,p_0')$. 
Let $\eps_1,\eps_2>0$, let $\nu>n\abs{\frac{1}{p}-\frac{1}{2}}$ and  $N>\frac{\nu}{m}$, and let $s,t>0$. Let $\psi_s$ and $\tilde{\psi}_t$ be smooth funtions on $[0,\infty)$ with
\begin{align} \label{schur-lemma-cond1}
	\abs{\partial_\xi^\beta \psi_s(\xi)}
			&\lesssim \min((s\abs{\xi})^{-\eps_2},(s\abs{\xi})^{N+\eps_1}) \abs{\xi}^{-\beta}, \\ \label{schur-lemma-cond2}
	\abs{\partial_\xi^\beta \tilde{\psi}_t(\xi)}
			&\lesssim \min((t\abs{\xi})^{-\eps_1},(t\abs{\xi})^{N+\eps_2}) \abs{\xi}^{-\beta} 
\end{align}
for all indices $0 \leq \beta \leq \lfloor \nu \rfloor +1$.
Then the operator $\psi_s(L)\tilde{\psi}_t(L)$ satisfies $L^p$ off-diagonal estimates of order $\nu$ in $\max(s,t)^{1/m}$, with an extra factor $\min\left(\left(\frac{s}{t}\right)^{\eps_1},\left(\frac{t}{s}\right)^{\eps_2}\right)$. 
\end{lem}

\begin{proof}
Since all assumptions are symmetric in $s$ and $t$, it suffices to consider the case $s<t$. 
We write $\psi_s(\xi)\tilde{\psi}_t(\xi) = \left(\frac{s}{t}\right)^{\eps_1} (s\xi)^{-\eps_1} \psi_s(\xi) (t\xi)^{\eps_1}\tilde{\psi}_t(\xi)$. Then observe that the functions $\xi \mapsto (s\xi)^{-\eps_1} \psi_s(\xi)$ and $\xi \mapsto (t\xi)^{\eps_1}\tilde{\psi}_t(\xi)$ satisfy the assumptions of Lemma \ref{lemma1}, and thus also the function $\xi \mapsto (s\xi)^{-\eps_1} \psi_s(\xi) (t\xi)^{\eps_1}\tilde{\psi}_t(\xi)$ satisfies the assumptions of Lemma \ref{lemma1} with respect to $t=\max(s,t)$. This gives the desired result.
\end{proof}

To pass from a multiplier to a pseudo-differential operator (with non-constant coefficients), we use the standard freezing argument of Coifman and Meyer, together with the 
Sobolev embedding given in Assumption \ref{ass:sobolev}.

\begin{lem} \label{lemma2}
Let $p \in (p_0,p_0')$. Let $\nu >n\abs{\frac{1}{p}-\frac{1}{2}}$ and $N>\frac{\nu}{m}$, let $r\in (0,1]$. Let $\sigma$ be a smooth function on $X \times (0,\infty)$ with
\begin{equation} \label{m-cond1-bis}
 \left|\Delta_x^\alpha \partial^\beta_\xi \sigma(x,\xi)  \right| \lesssim \min(1,(r^m|\xi|)^{N}) |\xi|^{-\beta}
\end{equation}
for all indices $\alpha,\beta$ with $0 \leq  \alpha \leq M_0$ and $0 \leq \beta \leq \lfloor \nu \rfloor +1$.  
Then there exists a constant $C>0$, independent of $r$, such that for all balls $B_1,B_2$ of radius $r$ and all $f \in L^p(X)$ with $\supp f \subseteq B_1$
\begin{align*}
 	\norm{\sigma(x,L)f}_{L^p(B_2)} \leq C \left(1+\frac{d(B_1,B_2)}{r}\right)^{-\nu} \norm{f}_{L^p(B_1)}.
\end{align*}
\end{lem}

\begin{proof}  Indeed, choose two balls $B_1,B_2$ of radius $r\in (0,1]$ and a function $f \in L^p(X)$, supported on $B_1$. Then we have for almost every $x\in B_ 2$
$$ \left| \sigma(x,L)(f)(x) \right| \leq \esssup_{y\in B_2} \left| \sigma(y,L)(f)(x) \right|.$$
According to Assumption \ref{ass:sobolev}, we know that 
$$ \left| \sigma(x,L)(f)(x) \right| \lesssim \sum_{i \geq 0} \sum_{j=0}^{M_0} 2^{-i\kappa} \mu(2^i B_2)^{-\frac{1}{p}}\left\| (r^2\Delta)^ j \sigma(\,\cdot\,,L)(f)(x) \right\|_{L^p(2^i B_ 2)},$$
where $M_0$ is the integer given in Assumption \ref{ass:sobolev}. This yields
$$ \left\| \sigma(x,L)(f) \right\|_{L^p(B_ 2)} \lesssim \sum_{i \geq 0}\sum_{j=0}^{M_0} 2^{-i\kappa}\mu(2^iB_2)^{-\frac{1}{p}} \left\| (y,x) \mapsto (r^2\Delta_y)^ j \sigma(y,L)(f)(x) \right\|_{L^p(2^iB_ 2\times B_2)}.$$
Fix $i \geq 0$, fix the point $y\in 2^iB_2$ and $j$, then $ (r^2\Delta_y)^j \sigma(y,L)$ is a multiplier, whose  symbol $ F(\xi) := (r^2\Delta_y)^ j \sigma(y,\xi)$ satisfies (\ref{m-cond1})
because of assumption (\ref{m-cond1-bis}) and the assumption $r \in (0,1]$. Consequently, we may apply Lemma \ref{lemma1} and obtain (uniformly in $y$ and $j$):
$$ \left\| (r^2\Delta_y)^ j \sigma(y,L)(f) \right\|_{L^p(B_2)} \lesssim \left(1+\frac{d(B_1,B_2)}{r}\right)^{-\nu} \|f\|_{L^p(B_1)} .$$
This estimate is uniform, so we may average it over $y\in 2^i B_2$ and sum over $i$ and $j$ in order to finally prove
\be{local-bis} \| \sigma(x,L) f \|_{L^p(B_2)} \lesssim \left(1+\frac{d(B_1,B_2)}{r}\right)^{-\nu} \|f\|_{L^p(B_1)}. \ee 
\end{proof}

In the Euclidean situation, it is well-known that pseudodifferential symbols can be split into elementary symbols \cite{CM}. This reduction is very abstract and we detail it in our context:

\begin{lem}[Decomposition into elementary symbols] \label{lem:decom} Let $\sigma \in S^0_{1,\delta}$ with $\delta\in[0,1]$. Then 
$$ \sigma = \tau + \sum_{l\in\Z} \sigma_l,$$
where $\tau \in S^0_{1,0}$, and $\sigma_l\in S^0_{1,\delta}$ is of the form
$$ \sigma_l(x,\xi) = \int_0^1 \gamma_{t,l}(x) \psi_{t,l}(\xi) \, \frac{dt}{t}$$
with smooth functions $\psi_{t,l}$, $\supp \psi_{t,l} \subseteq [t^{-1},2t^{-1}]$ and
$$ \|\sigma_{l}\|_{S^0_{1,\delta}} \lesssim (1+|l|)^{-M}$$
for every integer $M>0$.
\end{lem}

\begin{proof} Let $\sigma \in S^0_{1,\delta}$ with $\delta\leq 1$. Let $\phi,\psi \in \mathcal{S}(\R)$ with $\supp \phi \subseteq [0,2]$ and $\supp \psi \subseteq [1,2]$ such that for every $\xi \in (0,\infty)$
$$ 1=\int_0^1 \psi(t\xi)\,\frac{dt}{t} + \phi(\xi).$$
We then have 
$$ \sigma(x,\xi) = \int_0^1  \sigma(x,\xi) \psi(t \xi) \frac{dt}{t} + \phi(\xi)\sigma(x,\xi) =: \sigma^1(x,\xi)+ \tau (x,\xi).$$
For the second quantity, we have $\tau (x,\xi) = \phi(\xi)\sigma(x,\xi) \in S^0_{1,0}$.
Let us now fix $x\in X$. For every $t\in(0,1)$, the function $\xi \to  \sigma(x,\xi) \psi(t\xi)$ has support in $[t^{-1},2t^{-1}]$ and can be extended to a $t^{-1}$-periodic function. Hence, we know that we can expand it as a Fourier series, which yields for $\xi \in [t^{-1},2t^{-1}]$
$$ \sigma(x,\xi) \psi(t \xi) = \sum_{l\in\Z} \gamma_{l,t}(x) e^{2i\pi l t \xi},$$
where 
$$ \gamma_{l,t}(x) := t \int_{t^{-1}}^{2t^{-1}} \sigma(x,\eta) \psi(t \eta) e^{-2i\pi l t \eta} d\eta.$$
Using the regularity assumption on the symbols, we obtain via integration by parts for every $l \in \Z \setminus \{0\}$ and every $M \in \N$
\begin{align*}
 \left| \Delta_x^{j} \gamma_{l,t}(x) \right| & \lesssim \left(\frac{1}{\abs{l}t}\right)^M t\int_{t^{-1}}^{2t^{-1}} \left|\partial_\eta^M \Delta_x^{j}  \sigma(x,\eta) \psi(t \eta)\right| d\eta \\
 & \lesssim \left(\frac{1}{\abs{l}t}\right)^M t^{M-\frac{2}{m}\delta j} \lesssim \abs{l}^{-M} t^{-\frac{2}{m} \delta j}.
\end{align*}
So by choosing another function $\tilde{\psi} \in \mathcal{S}(\R)$ (with $\tilde{\psi}=1$ on the support of $\psi$ and $\tilde \psi$ still supported on $[1,2]$), we have
$$ \sigma^1(x,\xi) = \sum_{l\in\Z} \sigma_l(x,\xi)$$
where the elementary symbols
$$ \sigma_l(x,\xi) := \int_0^1 \gamma_{l,t}(x) \tilde{\psi}(t\xi) e^{2i\pi lt\xi} \frac{dt}{t}$$
have the expected properties.
\end{proof}

The following proposition yields $L^p$ off-diagonal estimates for approximations $T_\sigma\tilde{\psi}_t(L)$ of pseudodifferential operators $T_\sigma$, when $\sigma$ is an elementary symbol. Here, $T_\sigma$ is defined via \eqref{def-psdo}.

\begin{prop} \label{prop:can}
Let $p \in (p_0,p_0')$.
 Consider an elementary symbol, as appeared in the previous lemma, of the form
$$ \sigma(x,\xi) = \int_0^1 \gamma_{s}(x) \psi_{s}(\xi) \, \frac{ds}{s},$$
with $\|\gamma_s\|_{L^\infty} \lesssim 1$ uniformly in $s$, 
and $\psi_s$ a smooth function with compact support, adapted to the scale $s^{-1}$.
Let $t>0$, and let $\tilde{\psi}_t$ be a smooth function on $[0,\infty)$, satisfying \eqref{schur-lemma-cond2} with $\nu>n\abs{\frac{1}{p}-\frac{1}{2}}$, $N>\frac{\nu}{m}$, $\eps_1>0$ and $\eps_2>\frac{\nu}{m}$. 
Then the operator $T_{\sigma} \tilde{\psi}_t(L)$ satisfies $L^p$ off-diagonal estimates of order $\nu$ at the scale $t^{\frac{1}{m}}$.
\end{prop}

\begin{proof}
Fix $t\in(0,1)$, and choose two balls $B_1,B_2$ in $X$.  Let $f \in L^p(X)$ with support in $B_1$. Using the uniform boundedness of $(\gamma_s)$, we have
\begin{align*}
	\norm{T_{\sigma} \tilde{\psi}_t(L)f}_{L^p(B_2)}
			& = \norm{ \int_0^1 \gamma_s(\cdot) \psi_s(L) \tilde{\psi}_t(L)f \,\frac{ds}{s}}_{L^p(B_2)}\\
			& \lesssim \int_0^t \norm{\psi_s(L)\tilde{\psi}_t(L)f}_{L^p(B_2)} \frac{ds}{s}
				+\int_t^1 \norm{\psi_s(L)\tilde{\psi}_t(L)f}_{L^p(B_2)} \frac{ds}{s}.
\end{align*}
Due to its support condition, $\psi_s$ satisfies \eqref{schur-lemma-cond1}. Thus, the application of Lemma \ref{schur-lemma} yields that the above is bounded by a constant times
\begin{align*}
	\int_0^t \left(1+\frac{d(B_1,B_2)}{t^{1/m}}\right)^{-\nu} \left(\frac{s}{t}\right)^{\eps_1}\,\frac{ds}{s} \norm{f}_{L^p(B_1)}
	+ \int_t^1 \left(1+\frac{d(B_1,B_2)}{s^{1/m}}\right)^{-\nu} \left(\frac{t}{s}\right)^{\eps_2}\,\frac{ds}{s} \norm{f}_{L^p(B_1)}.
\end{align*}
This gives the estimate for $d(B_1,B_2)<t$. For $d(B_1,B_2)>t$, one can estimate the integral over $(t,1)$ against
\begin{align*}
	\int_t^1 \left(\frac{d(B_1,B_2)}{t^{1/m}}\right)^{-\nu} \left(\frac{t}{s}\right)^{\eps_2-\frac{\nu}{m}} \,\frac{ds}{s} \norm{f}_{L^p(B_1)}.
\end{align*}
Since we assumed $\eps_2>\frac{\nu}{m}$, we finally obtain
\begin{align*}
	\norm{T_{\sigma} \tilde{\psi}_t(L)f}_{L^p(B_2)}
				& \lesssim \left(1+\frac{d(B_1,B_2)}{t^{1/m}}\right)^{-\nu}\norm{f}_{L^p(B_1)}.
\end{align*}
\end{proof}

\begin{lem}[Extrapolation lemma] \label{lemma:extrapolation}
Let $T$ be a linear operator, bounded on $L^2$. If for every $t>0$ and some function $\psi_t$ satisfying (\ref{m-cond1}), $T \psi_t(L)$ satisfies $L^2-L^2$ off-diagonal estimates at the scale $t^{\frac{1}{m}}$ ; then $T$ is bounded on $L^p$ for every $p\in(p_0,2]$.
\end{lem}

\begin{proof} We just sketch the proof. Using the assumed off-diagonal estimates, we still have off-diagonal estimates with the $\psi_t$ function replaced by any $\phi_t$ function, satisfying only decay at $0$ (see \cite[Lemma 4.12]{FreyKunstmann} or \cite[Corollary 3.6]{B} in a specific case). Then applying with the particular function $\phi_t(\cdot):= 1-e^{-t \cdot}$, we can apply the extrapolation result of \cite{Bernicot},\cite[Theorem 5.11]{BJ} and obtain $L^p$-boundedness of the operator $T$.
\end{proof}

\begin{rem} Using \cite[Corollary 4.13]{FreyKunstmann}, we also have that such an operator (as in the previous lemma) satisfies boundedness on Hardy space $H^p_{L}$ into $L^p$ for $p\in[1,2)$. We regain the previous lemma, since for $p\in(p_0,2)$, the Hardy space $H^p_L$ is shown to be equal to $L^p$.
  
Moreover from \cite[Theorem 6.4]{BJ}, we also know that such operators satisfy weighted boundedness: an operator $T$ has in lemma \ref{lemma:extrapolation} is bounded on $L^p(\omega)$ for every $p\in(p_0,2)$ and $\omega \in {\mathbb A}_{\frac{p}{p_0}} \cap RH_{\left(\frac{2}{p}\right)'}$.
\end{rem}

\section{Symbols in $S^0_{1,0}$}

\begin{thm} \label{thm:s10} Assume that the ambiant space satisfies the Sobolev inequality: Assumption \ref{ass:sobolev}.
Let $p\in(p_0,p_0')$. Then, every symbol $\sigma\in S^0_{1,0}$ gives rise to a bounded operator on $L^p(X)$. More precisely, such an operator satisfies $L^p$ off-diagonal estimates of arbitrary order at the scale $1$.
\end{thm}

\begin{proof}
Consider a smooth symbol $\sigma$, belonging to the class $S^0_{1,0}$. We want to check that the operator $\sigma(x,L)$ satisfies off-diagonal estimates at the scale $1$. Let us fix  two balls $B_1,B_2$ of radius $1$. \\
If $d(B_1,B_2) \leq 10$, then we only use $L^p$ boundedness. Indeed, we know from spectral multiplier theory (see e.g. \cite{KunstmannUhl}, Theorem 5.4 b)) that every bounded function $F$ gives rise to a $L^p$-bounded linear operator, which is in particular bounded from $L^p(B_1)$ to $L^p(B_2)$. Then applying similar arguments as in Lemma \ref{lemma2} (freezing Coifman-Meyer argument), we can extend this boundedness to every operator coming from a smooth symbol $\tau\in W^{s,\infty}_x(L^\infty)$ for a large enough $s>0$. Obviously, $S^0_{1,0} \subset W^{s,\infty}_x(L^\infty)$, so we conclude that our operator $\sigma(x,L)$ is bounded from $L^p(B_1)$ to $L^p(B_2)$ with a norm
$$ \|\sigma(x,L)\|_{L^p(B_1) \to L^p(B_2)} \lesssim 1 \simeq \left(1+d(B_1,B_2)\right)^{-M}$$
since $d(B_1,B_2)\leq 10$. \\
So let us now focus on the main interesting case: $d(B_1,B_2)\geq 10$. Choose a smooth function $\psi$ (supported on $[1,2]$) and another one $\phi$ supported on $[0,2]$ such that for every $\xi\in (0,\infty)$
$$ 1 = \int_0^1 \psi(t\xi) \frac{dt}{t} + \phi(\xi).$$
We also split the symbol $\sigma$ with
$$ \sigma(x,\xi) = \int_0^1  \sigma(x,\xi) \psi(t\xi) \frac{dt}{t} + \sigma(x,\xi)\phi(\xi) = \sigma_1(x,\xi)+ \sigma(x,\xi)\phi(\xi).$$
The symbol $\sigma(x,\xi)\phi(\xi)$ satisfies \eqref{m-cond1-bis} with $r=1$. Thus, Lemma \ref{lemma2} yields that the second operator already satisfies  off-diagonal estimates (of arbitrary order) at the scale $1$.
 Moreover, for $t\in(0,1)$, the operator $\sigma(x,L) \psi(tL)$ satisfies off-diagonal estimates of arbitrary order $\nu>0$ at the scale $t^{1/m}$ (since the symbol verifies (\ref{m-cond1-bis}) and Lemma \ref{lemma2}).
This also implies that (see e.g. \cite{FreyKunstmann}, Remark 3.8)
\begin{align*}
	\norm{\sigma(x,L) \psi_t(L)f}_{L^p(B_2)}
		\lesssim \left(\frac{1}{t}\right)^{n/m} \left(1+\frac{d(B_1,B_2)}{t^{1/m}}\right)^{-\nu} \norm{f}_{L^p(B_1)}.
\end{align*}
Hence, for $d(B_1,B_2)\geq 10$ and $\nu>n$,
\begin{align*} 
\left\|\sigma(x,L)(f)\right\|_{L^p(B_2)} 
& \lesssim   \| f \|_{L^p(B_1)} \left( \int_{0}^{1} \left(1+\frac{d(B_1,B_2)}{t^{1/m}}\right)^{-\nu} \left(\frac{1}{t}\right)^{n/m} \frac{dt}{t}\right)  \nonumber \\
& \lesssim \left(1+d(B_1,B_2)\right)^{-\nu}  \|f\|_{L^p(B_1)}. \label{eq:sollu}
\end{align*}
As a consequence, it follows that $\sigma(x,L)$ admits off-diagonal decay at the scale $1$, and it is well-known that such operators are globally bounded on $L^p(X)$, if the order of the off-diagonal estimates is large enough (by splitting the whole space with an almost-disjoint covering by balls of radius $1$).
\end{proof}

For $s>0$ and $p\in(p_0,p_0')$, we may define the Sobolev space $W^{s,p}_{L}$ by the Bessel potential: $W^{s,p}_{L}$ is the closure of
$$ \{ f\in {\mathcal D}(L^s),\ \|f\|_{W^{m,p}_L}:=\|(1+L)^{\frac{s}{m}} f\|_{L^p} <\infty\}$$
for the corresponding norm. Similarly for operator $\Delta$. This is the classical way to define Sobolev spaces adpated to a semigroup of operators, see \cite[Section 8.4]{HofmannMayborodaMcIntosh}, \cite{BBR} for some properties ....).

\begin{cor} \label{cor1} Assume that the ambiant space satisfies the Sobolev inequality: Assumption \ref{ass:sobolev}. 
Let $p\in(p_0,p_0')$ and a symbol $\sigma\in S^0_{1,0}$. Then for $s>0$, $\sigma(x,L)$ is bounded from $ W^{s,p}_{L}$ to $W^{s,p}_{\Delta}$.
\end{cor}

\begin{proof} The boundedness of $\sigma(x,L)$ from $W^{s,p}_{L}$ to $W^{s,p}_{\Delta}$ is equivalent to the $L^p$-boundedness of $T:= (1+\Delta)^{\frac{s}{2}} \sigma(x,L) (1+L)^{-\frac{s}{m}}$. We may rewrite $T=T_\tau$ as the pseudodifferential operator associated to the symbol
$$ \tau(x,\xi):= (1+\Delta)^{\frac{s}{2}} \sigma(x, \xi) (1+\xi)^{-\frac{s}{m}}.$$
Since $\sigma \in S^0_{1,0}$, we let it to the reader to check that $\tau \in S^0_{1,0}$. So the previous theorem implies the $L^p$-boundedness of $T$.
 \end{proof}

Similarly, we have 

\begin{cor} \label{cor1-bis} Assume that the ambiant space satisfies the Sobolev inequality: Assumption \ref{ass:sobolev}. 
Let $p\in(p_0,p_0')$ and a symbol $\sigma\in S^m_{1,0}$ for some $m>0$. Then for $s>0$, $\sigma(x,L)$ is bounded from $ W^{s+m,p}_{L}$ to $W^{s,p}_{\Delta}$.
\end{cor}

\section{Symbols in $S^0_{1,\delta}$ with $\delta<1$}

\begin{thm} \label{thm:s1delta} Assume that the ambiant space satisfies the Sobolev inequality Assumption \ref{ass:sobolev} with the weak Poincar\'e inequality Assumption \ref{ass:poincare}.
Every symbol $\sigma\in S^0_{1,\delta}$ with $\delta<1$ gives rise to a $L^2$-bounded operator, still satisfying off-diagonal estimates at the scale $1$.
\end{thm}

\begin{proof}
According to Lemma \ref{lem:decom}, it suffices to prove Theorem \ref{thm:s1delta} for elementary symbols only. So let us just consider a symbol of the following form:
$$ \sigma(x,\xi) = \int_0^1 \gamma_{t}(x) \psi_{t}(\xi) \frac{dt}{t}$$
with smooth functions $\psi_{t}$ supported around $t^{-1}$ and $\gamma_t$ satisfying
\begin{align*}
 \left| \left(\Delta^j \gamma_{t}\right) \left( \partial_\xi^k \psi_{t} \right) \right| \lesssim n^{-M} t^{k-\frac{2}{m} \delta j},
\end{align*}
for sufficiently large indices $j,k$.
We will check that the adjoint operator $\sigma(x,L)^*$ is $L^2$-bounded, which is equivalent.
The adjoint operator is given by 
\begin{align}
T(f) := \sigma(x,L)^*(f) = \int_0^1 \overline{\psi_t}(L)[\overline{\gamma_t(.)} f] \frac{dt}{t}.
\end{align}
By orthogonality, we have
$$ \|T(f)\|_{L^2} \lesssim \| \overline{\psi_t}(L)[\overline{\gamma_t} f] \|_{L^2(\frac{d\mu dt}{t})}.$$
To each $t\in(0,1)$, let us fix an almost-disjoint covering $(Q_l^t)_l$ of $X$ by balls of radius $t^{\frac{1}{m}}$ and write for $x\in Q_l^t$
\begin{align} \label{eq:s1delta}
\overline{\psi_t}(L)[\overline{\gamma_t} f](x) = \overline{\psi_t}(L) \left[\left(\overline{\gamma_t} - \aver{Q_l^t}\overline{\gamma_t} \right) f \right](x) + \left(\aver{Q_l^t} \overline{\gamma_t} \right) \overline{\psi_t}(L)(f)(x).
\end{align}
The second quantity in \eqref{eq:s1delta} is estimated as follows (using the boundedness of $\gamma_t$):
$$ \left\| \| \left(\aver{Q_l^t} \overline{\gamma_t} \right) \overline{\psi_t}(L)(f)(x)\|_{L^2(Q_l^t)} \right\|_{\ell^2(l)} \lesssim \| \overline{\psi_t}(L)(f) \|_{L^2}$$
which gives
$$ \left\| \left\| \left\|  \left(\aver{Q_l^t} \overline{\gamma_t} \right) \overline{\psi_t}(L)(f) \right\|_{L^2(Q_l^t)} \right\|_{\ell^2(l)} \right\|_{L^2(\frac{dt}{t})} \lesssim \|f\|_{L^2},$$
by functional calculus.
Using $L^2-L^2$ off-diagonal estimates at the scale $r:=t^{\frac{1}{m}}$, the first quantity in \eqref{eq:s1delta} is estimated by
\begin{align*}
\left\| \overline{\psi_t}(L) \left[ \left(\overline{\gamma_t} - \aver{Q_l^t} \overline{\gamma_t} \right) f \right] \right\|_{L^2(Q_l^t)} & \lesssim \sum_{k} \left(1+\frac{d(Q_l^t,Q_k^t)}{r}\right)^{-M} \left\|  \left(\overline{\gamma_t} - \aver{Q_l^t} \overline{\gamma_t} \right) f \right\|_{L^2(Q_k^t)} \\
& \lesssim \sum_{k} \left(1+\frac{d(Q_l^t,Q_k^t)}{r}\right)^{-M+1} t^{\frac{1-\delta}{m}}  \| f\|_{L^2(Q_k^t)},
\end{align*}
where we used that 
\begin{equation} \left\|\overline{\gamma_t} - \aver{Q_l^t} \overline{\gamma_t} \right\|_{L^\infty(Q_k^t)} \lesssim d(Q_l^t,Q_k^t) t^{-\frac{\delta}{m}}. \label{eq:lip} \end{equation}
Indeed, to check this last estimate, we cover the geodesic joining the center of $Q_l^t$ and the one of $Q_k^t$ by a collection of balls $(O_i)_{i=1...,A}$ (with $A \simeq \frac{d(Q_l^t,Q_k^t)}{r}$). Using Assumption \ref{ass:poincare}, we then obtain
\begin{align*}
 \left\|\overline{\gamma_t} - \aver{Q_l^t} \overline{\gamma_t} \right\|_{L^\infty(Q_k^t)} & = \left\|\overline{\gamma_t} - \aver{O_1} \overline{\gamma_t} \right\|_{L^\infty(O_A)} \\
 & \leq \sum_{i=1}^A \left\|\overline{\gamma_t} - \aver{O_i}\overline{\gamma_t} \right\|_{L^\infty(O_{i+1})} \\
 & \lesssim \sum_{i=1}^A r (1+r^{2M(1-\delta)})^M r^{-\delta} \lesssim d(Q_l^t,Q_k^t) t^{-\frac{1}{m}(\delta)}.
\end{align*}
Then using the homogeneous type, we can sum over $l$ and we get 
$$ \left\| \| \overline{\psi_t}(L)[(\overline{\gamma_t} - \aver{Q_l^t}\overline{\gamma_t}) f]\|_{L^2(Q_l^t)} \right\|_{\ell^2(l)} \lesssim t^{\frac{1-\delta}{m}}  \| f\|_{L^2},$$
which is then integrable for $t\in(0,1)$, since $\delta<1$.
We also conclude that $T$ is bounded on $L^2$, which by duality gives the $L^2$-boundedness of $T^*=\sigma(x,L)$. \\
Having obtained this global boundedness, it remains to check the local off-diagonal estimates. Indeed, the proof of such inequalities (for two balls $Q_1,Q_2$ with $d(Q_1,Q_2) \geq 10$) in Theorem \ref{thm:s10} still holds in our current situation. We let it to the reader to check that the proof remains true: we loose some power of $t$ (when we use Sobolev estimate), but this is not a problem since this can be compensated by the fast off-diagonal decays.
\end{proof}

\begin{prop} \label{prop:dual}
Let $\sigma$ be a symbol $\sigma\in S^0_{1,\delta}$ with $\delta<1$. Then $ T:=[\sigma(x,L)]^*- \overline{\sigma}(x,L)$ is an operator bounded on $L^p$ for every $p\in (p_0,p_0')$.
\end{prop}

\begin{proof}
First, using Lemma \ref{lem:decom} and Theorem \ref{thm:s10}, it suffices to prove this current proposition for elementary symbols. So let us just consider a symbol of the following form:
$$ \sigma(x,\xi) = \int_0^1 \gamma_{t}(x) \psi_{t}(\xi) \frac{dt}{t}.$$
So 
$$ \overline{\sigma}(x,\xi) = \int_0^1 \overline{\gamma_{t}(x) \psi_{t}(\xi)} \frac{dt}{t},$$
and let
$$ T:= [\sigma(x,L)]^*- \overline{\sigma}(x,L) = \int_0^1 \overline{\psi_t}(L)[\overline{\gamma_{t}} \cdot ]  - \overline{\gamma_{t}}(x) \overline{\psi_{t}}(L) \frac{dt}{t}.$$
For every $t\in(0,1)$, let 
$$ T_t(f):=\overline{\psi_t}(L)[\overline{\gamma_{t}} f ]  - \overline{\gamma_{t}}(x)\overline{\psi_{t}}(L)f.$$
We are looking for  $L^p-L^p$ off-diagonal estimates at the scale $r:=t^{\frac{1}{m}}$, so consider an almost-disjoint covering $(Q_l^t)_l$ of $X$ by balls of radius $r$. Then
\begin{align*}
\left\| T_t(f) \right\|_{L^p(Q_l^t)} & \leq \sum_k \left\| \overline{\psi_t}(L)[(\overline{\gamma_{t}}-\aver{Q_k^t}\overline{\gamma_t}) f {\bf 1}_{Q_k^t} ]  \right\|_{L^p(Q_l^t)}+ \left\| \left(\overline{\gamma_{t}}(x)-\aver{Q_k^t}\overline{\gamma_t}\right) \overline{\psi_{t}}(L)[f {\bf 1}_{Q_k^t}]\right\|_{L^p(Q_l^t)} \\
& \lesssim \sum_k \left(1+ \frac{d(Q_l^t, Q_k^t) }{r}\right)^{-M}  d(Q_l^t,Q_k^t) t^{-\frac{\delta}{m}} \left\|  f \right\|_{L^p(Q_k^t)},
\end{align*}
where we used (\ref{eq:lip}).
As a consequence, we deduce that
\begin{align*}
\left\| T_t(f) \right\|_{L^p(Q_l^t)} \lesssim \sum_k \left(1+ \frac{d(Q_l^t, Q_k^t) }{r}\right)^{-M+1} t^{\frac{1-\delta}{m}} \left\|  f \right\|_{L^p(Q_k^t)},
\end{align*}
which yields, by using the homogeneous type of the ambiant space,
$$ \left\| T_t(f) \right\|_{L^p \to L^p} \lesssim t^{\frac{1-\delta}{m}}.$$
This can be integrated for $t\in(0,1)$, which finishes the proof of the $L^p$-boundedness of $T$.
\end{proof}

As a byproduct of Theorem \ref{thm:s1delta} and Proposition \ref{prop:dual}, we obtain

\begin{thm} \label{thm:s1delta-Lp}
Let $p\in(p_0,p_0')$.
Every symbol $\sigma\in S^0_{1,\delta}$ with $\delta<1$ gives rise to a $L^p$-bounded operator, still satisfying off-diagonal estimates at the scale $1$.
\end{thm}

\begin{proof}
For such a symbol $\sigma$, Theorem \ref{thm:s1delta} implies that $T_\sigma$ is $L^2$-bounded. Then, removing the $S^0_{1,0}$-part 
in the decomposition due to Lemma \ref{lem:decom} and Theorem \ref{thm:s10}, we may assume that $\sigma$ is an elementary symbol. 
Since $T_\sigma$ is bounded on $L^2$, because of Proposition \ref{prop:can} we may apply Lemma \ref{lemma:extrapolation} which gives that $T_\sigma$ is bounded on $L^p$ for every $p\in(p_0,2]$. For $p\in (2,p_0')$, we conclude by duality with Proposition \ref{prop:dual}.
\end{proof}

As for Corollaries \ref{cor1} and \ref{cor1-bis}, we have the following consequence:

\begin{cor} \label{cor2} 
Let $p\in(p_0,p_0')$ and a symbol $\sigma\in S^m_{1,\delta}$ for $\delta\in(0,1)$ and $m\geq 0$. Then for $s>0$ and under Assumptions \ref{ass:sobolev} and \ref{ass:poincare}, $\sigma(x,L)$ is bounded from $ W^{s+m,p}_{L}$ to $W^{s,p}_{\Delta}$.
\end{cor}

\section{The limit case $S^0_{1,1}$}

For the study of operators coming from a symbol, that belongs to the limit class $S^0_{1,1}$, we will restrict our attention to a more specific framework. We assume that $X$ is a Riemannian manifold, equipped with a metric $d$ and a doubling measure $\mu$. 
For a symbol $\sigma(x,\xi)=a(x) b(\xi)$ belonging to a class of type $S^0_{1,1}$, by definition the smoothness of $a$ and the smoothness of $b(L)$ (described by the decay at infinity of $b$) are of the same importance. This is specific to this case, and this brings the difficulty for the study of such symbols.

Consequently, it is quite natural to assume that the regularity in $x$ will be measured by the same operator $L$ \footnote{It is probably possible to extend the next results in a more general framework with different operators $H$, $L$ with some commutativity assumptions. Here, we prefer to focus on this simpler situation for convenience.}. So we deal with the same operator $L=H=\Delta$ given by the Laplace-Beltrami operator, coming from the Riemannian structure. \\
Moreover, we will assume that this operator is a sub-Laplacian operator $ H=L=\Delta= -\sum_{i=1}^\kappa X_i^2$ (of order $2$), given by a finite collection of vector fields $(X_i)_{i}$. For $I:=(i_1,...,i_p) \in\{1,..,\kappa\}^p$ a sequence of indices, we set $X_I:=X_{i_1} \cdots X_{i_p}$ the composition of $p$ vector fields, which is also an operator of order $p$.
Then, we still require Assumption \ref{ass-Lq-offidag}:

\begin{ass} \label{ass1} We assume that the analytic semigroup $(e^{-tL})_{t>0}$, generated by $-L$, satisfies $L^{p_0}-L^{p_0'}$ off-diagonal estimates (\ref{eq:Lp0}) for some $p_0\in[1,2)$.
\end{ass}

Assumption \ref{ass1} implies full $L^2-L^2$ off-diagonal estimates for the semigroup $(e^{-tL})_{t>0}$ (see e.g. the proof of \cite{AuscherMartell2}, Proposition 3.2 b)): there exists constants $C,c>0$ such that for all open subsets $E,F$ of $X$
\begin{equation} \label{eq:full-DG}
 	\|e^{-tL} \|_{L^{2}(E) \to L^{2}(F)} \leq C  \ e^{-c\left(\frac{d(E,F)^2}{t}\right)}.
\end{equation}

We assume some $L^2-L^2$ estimates for the ``gradients'' of the heat semigroup:

\begin{ass}[$L^{2}-L^{2}$ off-diagonal estimates] \label{ass2} 
For an arbitrary collection $I$ of indices, there exist constants $C,c>0$ such that for arbitrary balls $B_1,B_2$ of radius $r=t^{1/2}>0$
\begin{equation} \label{eq:L2}
 	\| r^{|I|} X_I e^{-r^2 L} \|_{L^{2}(B_1) \to L^{2}(B_2)} \leq C \ e^{-c\left(\frac{d(B_1,B_2)^2}{t}\right)}.
\end{equation}
\end{ass}

Finally, our last assumption is a Poincar\'e type inequality:

\begin{ass} \label{ass3} The manifold $(X,d,\mu)$ satisfies a $L^2$ Poincar\'e inequality: for every $f\in {\mathcal D}(L)$ and every ball $B$ of radius $r>0$,
$$ \left( \aver{B} \left|f- \aver{B} f \right|^2 d\mu \right)^{\frac{1}{2}} \lesssim r \left( \aver{B} |X(f)|^2 d\mu \right)^{\frac{1}{2}}.$$ 
\end{ass}

We can then define the symbols:

\begin{df} We set $S^{0}_{1,1}:=S^{0}_{1,1}(L)$ the set of symbols $\sigma\in C^\infty(X \times (0,\infty))$ such that
 $$ \forall I, \ \forall \beta\geq 0, \qquad \left| \partial_\xi^\beta X_I \sigma(\,.\,,\xi) [x] \right| \lesssim (1+|\xi|)^{-\beta + |I|}.$$
\end{df}

Then, following Lemma \ref{lemma1}, we have this new version taking into account the setting of the sub-Laplacian structure:

\begin{lem} \label{lemma-bis}
Let $\nu >0$ and $N>\nu$, let $r>0$. Take $F$ a smooth function on $(0,\infty)$ with
\begin{equation} \label{m-cond1-bi}
 \left|\partial^\beta_\xi F(\xi)  \right| \lesssim \min(1,(r^2|\xi|)^N) |\xi|^{-\beta}
\end{equation}
for all indices $0 \leq 2 \beta \leq \lfloor \nu \rfloor +1$. 
Then for all set of indices $I$, there exists a constant $C>0$, independent of $r=t^{1/2}>0$, such that for all balls $B_1,B_2$ and all $f \in L^2(X)$ with $\supp f \subseteq B_1$
\begin{align*}
 	\norm{r^{|I|} X_I F(L)f}_{L^2(B_2)} \leq C \left(1+\frac{d(B_1,B_2)}{r}\right)^{-\nu} \norm{f}_{L^2(B_1)}.
\end{align*}
\end{lem}

\begin{proof}
The proof is very similar to the one used for Lemma \ref{lemma1}, with using the following off-diagonal estimates: ($z=r^2+it$)
\begin{equation} \| r^{|I|}X_I e^{-(r^2+it)L} \|_{L^2(B_2)} \lesssim \exp\left(-c\frac{d(B_1,B_2)^2}{r^2}\right) \left(1+ \frac{t}{r}\right)^{\frac{\nu}{2}} \norm{f}_{L^2(B_1)}. \label{eq:e} \end{equation}
The considered operator can be split as follows:
 $$ r^{|I|}X_I e^{-(r^2+it)L} = \left(r^{|I|}X_I e^{-\frac{r^2}{2} L}\right) \left(e^{-(\frac{r^2}{2}+it)L}\right). $$
The first operator $\left(r^{|I|}X_I e^{-\frac{r^2}{2} L}\right)$ satisfies $L^2-L^2$ off-diagonal estimates due to Assumption \ref{ass2} and the second operator $\left(e^{-(\frac{r^2}{2}+it)L}\right)$ (already used in Lemma \ref{lemma1}) satisfies the following off-diagonal estimates (see \cite[Lemma 4.9]{KunstmannUhl}) 
$$ \norm{e^{-(\frac{r^2}{2}+it)L}}_{L^2(B_2)} \lesssim \left(1+\frac{d(B_1,B_2)}{r}\right)^{-\nu} \left(1+ \frac{t}{r}\right)^{\frac{\nu}{2}}\norm{f}_{L^2(B_1)}.$$
Computing the composition of these two off-diagonal estimates, we obtain (\ref{eq:e}).
Then as explained in \cite{KunstmannUhl}, we expand $F$ with its Fourier transform. The smoothness of $F$ allows us to compensate for the weight $\left(1+ \frac{t}{r}\right)^{\frac{\nu}{2}}$ and we may obtain the stated result from (\ref{eq:e}).
\end{proof}

Then, it remains us to define the space $\BMO_L$ associated with the sub-Laplacian operator (see \cite{DY,HLMMY} and the references given therein for more details):

\begin{df} For $\eps>0$, $M \in \N$ and $\phi \in {\mathcal R}(L^M) \subseteq L^2(X)$ (i.e. there exists $b\in L^2 \cap {\mathcal D} (L^M)$ with $\phi=L^M(b) \in L^2$), we introduce the norm
$$ \norm{\phi}_{*,M} := \sup_{j \geq 0} \left[2^{j\eps} \mu(2^j B_0)^{1/2} \sum_{k=0}^M \|L^k b\|_{L^2(S_j(B_0))}\right], $$
where $B_0$ is a fixed ball with radius 1. Then
\begin{align} \label{def-molMeps}
 	{\mathcal M}_{M}:=\{\phi \in {\mathcal R}(L^M) \,:\, \norm{\phi}_{*,M} < \infty\}.
\end{align}
\end{df}

Then for every $f \in ({\mathcal M}_{M})'$ and every $t>0$, one can define $(I-e^{-tL})^M f$ and $(I-(I+tL)^{-1})^Mf$ via duality as elements of $L^2_{\textrm{loc}}(X)$.

\begin{df}[$\BMO_L$ space]
 Let  $M \in \N$. An element $f \in {\mathcal M}_M$ is said to belong to $\BMO_{L,M}(X)$ if
\begin{equation} \label{defBMO-1}
 	\norm{f}_{\BMO_{L,M}(X)} := 
		\sup_{B \subseteq X} \left(\aver{B} \left|(I-e^{-r_B^{2}L})^M f \right|^2 d\mu \right)^{1/2} < \infty,
\end{equation}
where the supremum is taken over all balls $B$ in $X$.
Then it is known that such spaces do not depend on $M$, as soon as $M>\frac{n}{4}$ (see \cite{HLMMY}), and we set $\BMO_L:=\BMO_{L,M}$. 
\end{df}

From the $L^2-L^2$ off-diagonal estimates on the semigroup, it is easy to check that $L^\infty \subset \BMO_L$.

\begin{thm} Let us suppose Assumptions \ref{ass1}, \ref{ass2} and \ref{ass3}. Let a symbol $\sigma \in S^0_{1,1}$. Then $T_\sigma:=\sigma(x,L)$ is $L^2$-bounded if and only if $T_\sigma^*(1) \in \BMO_{L}$. 
\end{thm}

As we explained in Lemma \ref{lemma:extrapolation}, one can show via extrapolation that such operators are also bounded on $L^p$ for $p\in(p_0,2)$ and in Sobolev spaces (as in Corollary \ref{cor1}).

\begin{proof}

The main idea is to apply the $T(1)$-theorem, adapted to the semigroup $(e^{-tL})_{t>0}$, proved in \cite[Theorem 4.16]{FreyKunstmann}. Let $\Psi(tL):=(tL)^{M} e^{-tL}$ (for a large enough integer $M$) and $\Phi(tL)=e^{-tL}$. Since $L$ is a sub-Laplacian operator, we have $L(1)=0$. Hence for $\sigma \in S^0_{1,1}$, we have $T_\sigma(1)=\sigma(\cdot,0) \in L^\infty \subset \BMO_L$.

So our result will be a consequence of \cite[Theorem 4.16]{FreyKunstmann}, as soon as we have checked the following estimates: for every $t>0$ every balls $B_1,B_2$ of radius $r:=t^{\frac{1}{2}}$
\begin{equation}
  \| \Psi(tL) T_{\sigma} \Phi(tL) \|_{L^2(B_1) \to L^2(B_2)} +  \| \Psi(tL) [T_{\sigma}]^* \Phi(tL) \|_{L^2(B_1) \to L^2(B_2)} \lesssim \left(1+\frac{d(B_1,B_2)}{r} \right)^{-N},
\label{eq:check}
\end{equation}
with $N$ sufficiently large.

\medskip
{\bf Step 1 :} Estimates for $\Psi(tL) T_{\sigma} \Phi(tL)$. 

Since we may expand $L^M$ in a finite sum of vector fields $X_I=X_{i_1}...X_{i_{2M}}$ of order $2M$, let us just consider one of them and so we have to estimate $ \Phi(tL) (t^MX_I) T_{\sigma} \Phi(tL)$. \\
Due to Lemma \ref{lem:decom}, we may reduce the problem to elementary operators. So assume that
$$ \sigma(x,\xi) = \int_0^1 \gamma_{s}(x) \psi_{s}(\xi) \frac{ds}{s}$$
with $\psi_{s}$ supported around $s^{-1}$ and regular at this scale and $\gamma_s$ satisfying
$$ \|X_J \gamma_s \|_{L^\infty} \lesssim s^{-|J|}.$$
First we have
\begin{align*}
  \Phi(tL) (t^M X_I) T_{\sigma} \Phi(tL)(f) & = \int_0^1 \Phi(tL) (t^M X_I)\left[\gamma_s(\cdot) \psi_s(L) \Phi(tL)(f) \right] \frac{ds}{s}  \\
  & = \sum_{I_1\sqcup I_2=I} \int_0^1 \Phi(tL)\left[ (t^{|I_1|/2} X_{I_1} \gamma_s) (t^{|I_2|/2} X_{I_2} \psi_s(L) \Phi(tL)(f))\right] \frac{ds}{s}
\end{align*}
where we used the Leibniz rule
$$ X_I(fg) = \sum_{I_1\sqcup I_2=I} X_{I_1}(f) X_{I_2}(g).$$
So choose now a bounded covering  $(Q_i)_i$ of balls with radius $\sqrt{t}$, fix $s\in(0,1)$ and a covering $I_1\sqcup I_2=I$. We have 
\begin{align*}
 \left\| \Phi(tL)\left[ (t^{|I_1|/2} X_{I_1} \gamma_s) (t^{|I_2|/2} X_{I_2} \psi_s(L) \Phi(tL)(f))\right] \right\|_{L^2(B_2)} & \\
 & \hspace{-8cm} \lesssim \sum_{i} \left(1+\frac{d(B_2,Q_i)}{\sqrt{t}}\right)^{-M} \left\|\left[ (t^{|I_1|/2} X_{I_1} \gamma_s) (t^{|I_2|/2} X_{I_2} \psi_s(L) \Phi(tL)(f))\right] \right\|_{L^2(Q_i)} \\
& \hspace{-8cm} \lesssim \sum_{i} \left(1+\frac{d(B_2,Q_i)}{\sqrt{t}}\right)^{-M} \left\| (t^{|I_1|/2} X_{I_1} \gamma_s)\right\|_{L^\infty} \left\|t^{|I_2|/2} X_{I_2} \psi_s(L) \Phi(tL)(f) \right\|_{L^2(Q_i)} \\
& \hspace{-8cm} \lesssim \sum_{i} \left(1+\frac{d(B_2,Q_i)}{\sqrt{t}}\right)^{-M} \left(\frac{t}{s}\right)^{\frac{|I_1|}{2}} \left\|t^{|I_2|/2} X_{I_2} \psi_s(L) \Phi(tL)(f) \right\|_{L^2(Q_i)},
\end{align*}
where we used the regularity of $\gamma_s$. Then $f$ is supposed to be supported on the other ball $B_1$. Since $\psi_s \Phi(t\cdot)$ satisfies
$$ \left|\partial^\alpha_\xi \psi_s(\xi)\Phi(t\xi)  \right| \lesssim {\bf 1}_{s \xi \simeq 1} s^\alpha \min\{ \frac{s}{t},1\}^{N}$$
for every integer $N>1$, we obtain from Lemma \ref{lemma-bis}
$$ \left\|t^{|I_2|/2} X_{I_2} \psi_s(L) \Phi(tL)(f)\right\|_{L^2(Q_i)} \lesssim  \left(\frac{t}{s}\right)^{\frac{|I_2|}{2}} \min\{ \frac{s}{t},1\}^{N} \left(1+\frac{d(B_1,Q_i)}{\sqrt{s}}\right)^{-N} \|f\|_{L^2(B_1)},$$
for every sufficiently large exponent $N\geq 1$. 
Finally, we deduce that for some large integers $M,M',N\geq 1$, we have
\begin{align*}
 \lefteqn{ \left\| \Phi(tL)\left[ (t^{|I_1|/2} X_{I_1} \gamma_s) (t^{|I_2|/2} X_{I_2} \psi_s(L) \Phi(tL)(f))\right]\right\|_{L^2(B_2)}} &  & \\
 & & \lesssim \sum_{i} \left(1+\frac{d(B_2,Q_i)}{\sqrt{t}}\right)^{-M} \left(\frac{t}{s}\right)^{M'}   \min\{ \frac{s}{t},1\}^{N} \left(1+\frac{d(B_1,Q_i)}{\sqrt{s}}\right)^{-N} \|f\|_{L^2(B_1)} \\
 & & \lesssim \left(1+\frac{d(B_2,B_1)}{\max\{\sqrt{t},\sqrt{s}\}}\right)^{-M} \left(\frac{t}{s}\right)^{M}  \min\{ \frac{s}{t},1\}^{N} \|f\|_{L^2(B_1)}.
\end{align*}
 Integrating this inequality along $s\in(0,t)$ and $s\in(t,1)$, for a good choice of $M,M',N$ according to the two situations, we obtain
\begin{align*}
 \lefteqn{\left\| \Phi(tL)\left[ (t^{|I_1|/2} X_{I_1} \gamma_s) (t^{|I_2|/2} X_{I_2} \psi_s(L) \Phi(tL)(f))\right] \right\|_{L^2(B_2)}} &  & \\
  & & \lesssim \left(1+\frac{d(B_2,B_1)}{\sqrt{t}}\right)^{-M} \|f\|_{L^2(B_1)},
\end{align*}
which corresponds to the first part of (\ref{eq:check}).

\mb
{\bf Step 2 :} Estimates for $\Psi(tL) [T_{\sigma}]^* \Phi(tL)$.

By duality, $L^2-L^2$ off-diagonal estimates for $\Psi(tL) [T_{\sigma}]^* \Phi(tL)$ are equivalent to $L^2-L^2$ off-diagonal estimates for $\Phi(tL) T_{\sigma} \Psi(tL)$.
In Proposition \ref{prop:can}, we have already seen that $T_{\sigma} \Psi(tL)$ satisfies such estimates, and so in particular does $\Phi(tL) T_{\sigma} \Psi(tL)$.

We can then apply the $T(1)$-theorem and conclude the proof.
\end{proof}

As an application, it is known that $S^0_{1,1}$ pseudodifferential operators appear when we use the paraproduct to ``linearize'' a nonlinear PDE (this was Bony's motivation to introduce his paraproduct in \cite{Bony}). We may adapt this approach in our current context: we denote $\Psi(tL):=(tL)^{M} e^{-tL}$ (for a large enough integer $M$) and $\Phi(tL)=e^{-tL}$. We can then define paraproducts of the form
$$ (f,g) \rightarrow \Pi_g(f) := \int_0^1 \phi(tL)(g) \psi(tL)(f) \frac{dt}{t}.$$
Such paraproducts are very similar to those introduced in \cite{B,BS,Frey}, but not identical. One can then check that $\Pi_g(f) = T_\sigma(f)$ is a pseudodifferential operator associated to the symbol
$$ \sigma(x,\xi) = \int_0^1 \Phi(tL)(g) (x) (t\xi)^{M} e^{-t\xi} \frac{dt}{t}.$$
If $p_0=1$, that means that the semigroup satisfies pointwise estimates then for $g\in L^\infty$, $\Phi(tL)(g)$ is uniformly bounded and then $\sigma \in S^0_{1,1}$. We also deduce that the paraproduct $\Pi_g$ is bounded on $L^p$ as soon as $g\in L^\infty$, as obtained in \cite{B,Frey} for the other paraproducts.

The study of such paraproducts, given by a semigroup of operators, is just beginning. So it is not clear today which definition (this one or the one in \cite{B,Frey}) is the best.

Here the novelty is that such operators fit into the class of pseudodifferential operators. So as in the Euclidean situation: we hope to define a suitable functional calculus in order to ``invert'' the paraproduct $\Pi_g$ locally around points where $g(x)\neq 0$.

\small{

}

\end{document}